\documentclass[11pt, reqno]{amsart}

\usepackage{amsmath, amssymb, amscd, amsthm, amsfonts}
\usepackage{graphicx}
\usepackage{hyperref}
\usepackage{xcolor}

\makeatletter
\newcommand*\bigcdot{\mathpalette\bigcdot@{.5}}
\newcommand*\bigcdot@[2]{\mathbin{\vcenter{\hbox{\scalebox{#2}{$\m@th#1\bullet$}}}}}
\makeatother

\newtheorem{theorem}{Theorem}[section]
\newtheorem{lemma}[theorem]{Lemma}

\newtheorem{corollary}[theorem]{Corollary}
\newtheorem{proposition}[theorem]{Proposition}
\theoremstyle{definition}

\newtheorem{remark}{Remark}
\newtheorem*{acknow}{Acknowledgments}
\newcommand{\divg}{\mathrm{div}}

\allowdisplaybreaks[1]

\begin{document}

\title[A partially overdetermined problem in convex cones]{A partially overdetermined problem for $p$-Laplace equation in convex cones}

\author{Hui Ma}
\address{Department of Mathematical Sciences, Tsinghua University,
Beijing 100084, P.R. China}
\email{ma-h@mail.tsinghua.edu.cn}

\author{Mingxuan Yang}
\address{Department of Mathematical Sciences, Tsinghua University,
Beijing 100084, P.R. China} 
\email{ymx20@mails.tsinghua.edu.cn}

\author{Jiabin Yin}
\address{School of Mathematics and Statistics, Guangxi Normal University,
Guilin 541004, P.R. China}
\email{jiabinyin@126.com}



\begin{abstract}
  We consider a partially overdetermined problem for the $p$-Laplace equation in a convex cone $\mathcal{C}$ intersected with the exterior of a smooth bounded domain $\overline{\Omega}$ in $\mathbb{R}^n$($n\geq2$). First, we establish the existence, regularity, and asymptotic behavior of a capacitary potential. Then, based on these properties of the potential, we use a $P$-function, the isoperimetric inequality, and the Heintze-Karcher type inequality in a convex cone to obtain a rigidity result under the assumption of orthogonal intersection.  
\end{abstract}


\maketitle

\section{Introduction}

 The investigation of symmetry in overdetermined boundary value problems has emerged as a significant area of study within PDE theory. One of the seminal symmetry results in this field was obtained by Serrin \cite{Se}, which is now a classic and influential result. The main techniques used to address such problems include the method of moving planes, as well as  Weinberger's approach \cite{Wei}, which is based on the maximum principle for the so-called $P$-function and Rellich-Pohozaev's integral identity.

For our purpose, we recall an overdetermined problem for a capacity in an exterior domain. The capacity of a smooth bounded domain $\Omega$ in $\mathbb{R}^{n}$ is defined as
$${\rm Cap}(\Omega)=\inf_{v}\left\{\frac{1}{2}
\int_{\mathbb{R}^n}|\nabla v|^2 ~d x : v\in C_{c}^{\infty}(\mathbb{R}^n)~, v\geq 1 \text{ in } \Omega\right\}.
$$
The minimizer for ${\rm Cap}(\Omega)$ is characterized by the capacitary potential $u$ satisfying
\begin{equation}\label{eqn:1.1'}
\left\{
\begin{aligned}
&\Delta u=0 \text{ in } \mathbb{R}^n\setminus\bar{\Omega},\\
&u=1  \text{ on }  \partial\Omega,\\
&u\rightarrow 0  \text{ as } |x|\rightarrow +\infty.
\end{aligned}\right.
\end{equation}

In \cite{R2}, Reichel studied an overdetermined problem for a nonhomogeneous Laplace equation $\Delta u+f(u,|\nabla u|)=0$. In particular, when the nonhomogeneous term $f=0$, assuming an overdetermined boundary condition
\begin{equation}\label{eqn:1.2'}|\nabla u|=C \text{ on } \partial\Omega,
\end{equation}
the result in \cite{R2} can imply that \eqref{eqn:1.1'} and \eqref{eqn:1.2'} admit a solution if and only if $\Omega$ is a ball.

There have been numerous investigations into various types of overdetermined boundary value problems. One such problem, concerning the $p$-Laplace equation, was studied by Reichel in \cite{R1}. He utilized the method of moving planes and obtained a result for a more general class of quasilinear equations, including $p$-Laplace equations in exterior domains.

Following the original scheme of Weinberger's argument, Garofalo and Sartori \cite{GS} and Poggesi \cite{P} proved Reichel's result for the $p$-Laplace equation by using the $P$-function, which was first used by Payne and Philippin \cite{PP} for the exterior problem.
The corresponding overdetermined problem for the anisotropic $p$-capacity in an exterior domain has been extended by Bianchini, Ciraolo, and Salani \cite{BCS} and Bianchini and Ciraolo \cite{BC}. They proved the symmetric results with the assumption that the domain is convex, using an integral method. 
Later, Xia and Yin \cite{XY} obtained the same results without the convexity condition, under assumptions on the regularity of the anisotropic norm. 
For the problems involving the $p$-Laplacian and anisotropic $p$-Laplacian in $\mathbb{R}^{n}$, interested readers may refer to \cite{CS, FK, FGK, GL, Wang-Xia}.

Regarding the cone case, Pacella and Tralli \cite{PT} first characterized the spherical sectors for the Laplace equation in an interior domain within a convex cone in Euclidean space. Subsequently, Ciraolo and Roncoroni \cite{CR} generalized that to more general elliptic equations in space forms, which include the $p$-Laplace equation on a sector-like domain in a convex cone in Euclidean space. More recently, Ciraolo and Li \cite{CL} obtained a rigidity result for the anisotropic $n$-Laplace equation in an exterior domain. Notably, their results were obtained under weak regularity assumptions about the solution, domain, and cone. Other overdetermined problems in cones can be referred to \cite{CFR, DPV, IPT, PT2}.

Let us recall the result for Euclidean norm in the paper \cite{CL}: under a prescribed logarithmic condition at infinity and an overdetermined condition, they proved a rigidity result by showing that the
existence of a solution for $n$-Laplace equation implies that the set $\mathcal{C}\cap\Omega$ must be the intersection of $\mathcal{C}$ and a ball. A natural question is then the following:

 \textit{Is there a similar rigidity result in an exterior domain with respect to the convex cone $\mathcal{C}$ when} $1<p<n$?

 \quad

We focus on the partially overdetermined problem for $p$-capacity ($1< p< n$) in convex cone $\mathcal{C}$.

Let $\mathcal{C}$ be an open cone in $\mathbb{R}^{n}$, where $n \geq 2$, with vertex at the origin $O$. Specifically, let $\omega$ be an open connected domain on the unit sphere $S^{n-1}$, then 
$$\mathcal{C} := \{tx : x \in \omega, t \in (0, +\infty)\}.$$

Given an open convex cone $\mathcal{C}$ in $\mathbb{R}^{n}$ $(n\geq 2)$ and $\Omega\subset \mathbb{R}^{n}$ is a smooth bounded domain containing the vertex of the convex cone $\mathcal{C}$. We denote 
$$\Sigma:=\partial\Omega\cap \mathcal{C},\quad\Omega^{\mathcal{C}}:=\mathcal{C}\backslash\overline{\Omega} \text{ and } \Gamma:=\partial{\Omega}^{\mathcal{C}}\backslash \overline{\Sigma}.$$
Moreover, we assume throughout this paper that ${\Omega}^{\mathcal{C}}$ is connected, $\partial \overline{\mathcal{C}}\backslash{O}$ is smooth, the $(n-1)$-dimensional Hausdorff measure $\mathcal{H}^{n-1}\left(\overline{\Sigma}\right)>0$.

Motivated by \cite{CL, HKM, H}, we introduce the $p$-capacity of the sector-like domain $\Omega\cap{\mathcal{C}}$ in the convex cone $\mathcal{C}$ for $1<p< n$ that is defined by
\begin{equation}\label{eqn:capacity}
\begin{split}
&{\rm Cap}_{p}(\Omega\cap{\mathcal{C}};\mathcal{C}):=\inf_{\phi}\left\{\frac{1}{p}\int_{{\mathcal{C}}}|\nabla\phi|^{p}~ dx: \phi\in \mathcal{A}\right\},\\
    \text{ where }& \mathcal{A}:=\left\{\phi\in W^{1,p}(\mathcal{C}):  \phi-f=\omega \chi_{{\mathcal{C}}} \text{ for }\omega\in W^{1,p}_{0}(\mathbb{R}^{n}\backslash \overline{\Omega}) 
   \right\}. 
\end{split}
\end{equation}
Here the function $f\in C^{\infty}_{c}(\mathbb{R}^{n})$, $0\leq f\leq 1$ and $f=1$ in a neighborhood of $\Omega.$

From Lemma \ref{existence} below, we know that there exists a unique weak solution $u$ to the problem 
\begin{equation}\label{eqn:1.1}
\left\{
\begin{aligned}
&\Delta_{p}u=0 \text{ in }  {\Omega}^{\mathcal{C}},\\
&u=1 \text{ on }  {\Sigma},\\
&\langle\nabla u(x), \nu_{\Gamma}\rangle=0  \text{ on } \Gamma,\\
&u(x)\rightarrow 0 \text{ as } |x|\rightarrow+\infty.
\end{aligned}\right.
\end{equation}
where $\nu_{\Gamma}$ is a unit exterior normal to $\Gamma$ and $\Delta_{p}$ denotes the $p$-Laplace operator defined by $\Delta_{p} u=\divg (|\nabla u|^{p-2}\nabla u).$

Moreover, the solution $u$ satisfies
 $${\rm Cap}_{p}(\Omega\cap{\mathcal{C}};\mathcal{C})=\frac{1}{p}\int_{\Omega^{\mathcal{C}}} |\nabla u|^{p} ~dx.$$
The function $u$ is called the $p$-capacitary potential in convex cone $\mathcal{C}$ associated with $\Omega\cap\mathcal{C}$. 

In addition, we denote $$a(\xi)=\frac{1}{p}(\nabla|\xi|^{p})(\xi), \quad
a_{ij,p}(\xi)=|\xi|^{p-2}\delta_{ij}+(p-2)|\xi|^{p-4}\xi_{i}\xi_{j},$$
then \begin{equation}\label{p-harmonic}
\Delta_{p}u=\divg (a(\nabla u))=a_{ij,p}u_{ij}=[|\nabla u|^{p-2}\delta_{ij}+(p-2)|\nabla u|^{p-4}u_{i}u_{j}]u_{ij}.
\end{equation}
And we say a function $u\in W_{l o c}^{1, p}\left(\overline{{\mathcal{C}}}\backslash \Omega\right)$ is a weak solution of (\ref{eqn:1.1}) in $\Omega^{\mathcal{C}}$ if $u=1$ on ${\Sigma}$ and $u(x)\rightarrow 0 \text{ as } |x|\rightarrow+\infty$ such that
$$
\int_{\Omega^{\mathcal{C}}}\langle a(\nabla u), \nabla \varphi\rangle ~d x=0
$$
for all $\varphi \in W^{1, p}\left(\Omega^{\mathcal{C}}\right)$ with $\varphi=0$ on $\Sigma$ and with bounded support.
 
 Note that we allow the function $\varphi$ to be nonzero on the boundary $\Gamma$.

We will study the problem \eqref{eqn:1.1} with the overdetermined boundary condition
\begin{equation}\label{Overdeterminedcondition}
    |\nabla u|=C \text{ on }\Sigma.
\end{equation}
for some positive constant $C$. 

\quad

Note that in the \cite[Proposition 1.1]{CL}, $\mathcal{C}$ was written as $\mathcal{C} = \tilde{\mathcal{C}}\times\mathbb{R}^{k}$, where $\tilde{\mathcal{C}}\subset\mathbb{R}^{n-k}$ is an open convex cone with vertex at the origin which contains no lines and $k \in\{0,\ldots, n\}$. But now due to the assumption that $\partial \overline{\mathcal{C}}\backslash \{O\}$ is smooth, the convex cone $\mathcal{C}$ can not be written as in that form. 
It is easy to calculate that such a solution exists in the following special case:
\begin{proposition}[Model]
Let $B_{R}(O)$ be a ball with radius $R=\dfrac{n-p}{p-1}C^{-1}$ centered at the origin. Then the function \begin{equation}\label{model}
    u(x)=\left(\frac{|x|}{R}\right)^{\frac{p-n}{p-1}}
\end{equation} is the unique solution to the problem (\ref{eqn:1.1}) and (\ref{Overdeterminedcondition}).
\end{proposition}

\quad

Our main result concerns the partially overdetermined problem in convex cone $\mathcal{C}$ with an assumption of orthogonal intersection.
\begin{theorem}\label{thm:1.2}
Let $n\geq 2$, $1<p<n$, $\Omega\subset\mathbb R^n$ be a bounded domain containing the vertex of the convex cone $\mathcal{C}$ with a boundary of class $ C^{2,\alpha}$ for $\alpha\in(0,1)$ and $\overline{\Sigma}=\partial \Omega\cap \overline{\mathcal{C}}$ be an embedded hypersurface with boundary. Moreover, we assume that $\overline{\Sigma}$ is perpendicular to $\partial \overline{\mathcal{C}}$ along $\partial{\overline{\Sigma}}.$  Then \eqref{eqn:1.1} and \eqref{Overdeterminedcondition} admit a weak solution $u$ if and only if $u$ is given by (\ref{model}) for the radius $R=\dfrac{n-p}{p-1}C^{-1}$ and $\overline{\Sigma} = \partial B_R(O) \cap \overline{\mathcal{C}}$.
\end{theorem}

\quad

\begin{remark}
(i) In the special case $\mathcal{C}=\mathbb{R}^{n}$, the boundary $\Gamma=\emptyset$ and the condition $\langle\nabla u, \nu_{\Gamma}\rangle=0$ on $\Gamma$ is directly satisfied. Meanwhile, the assumption of orthogonal intersection is also not necessary, as we can directly apply the classic regularity results to obtain the solution $u\in C^{1}(\mathbb{R}^{n}\backslash {\Omega})$. Then it reduces to \cite[Therorem 1]{R1} (for $p$-Laplace equation) and \cite[Theorem 1]{XY}(for the anisotropic norm is equal to Euclidean norm, i.e., $F(\xi)=|\xi|$ for $\xi\in \mathbb{R}^{n}$).

(ii) By the smoothness of the boundary and the classic regularity results for the $p$-Laplace equation, we know that the weak solution $u\in C^{1}(\Omega^{\mathcal{C}}\cup\Sigma\cup\Gamma)$. However, the global regularity of $u$ is complicated. So we assume that $\Sigma$ is perpendicular to $\partial \overline{\mathcal{C}}$ along $\partial{\overline{\Sigma}}$. This assumption implies further regularity of the solution $u$, as shown in Theorem \ref{regularity}. 
\end{remark}
\quad

Without the assumption of orthogonal intersection, we can still prove the rigidity result under a regularity assumption.

\begin{corollary}\label{coro:1.3}
 Let $n\geq 2$, $1<p<n$, $\Omega\subset\mathbb R^n$ be a bounded domain containing the vertex of the convex cone $\mathcal{C}$ with a boundary of class $ C^{2,\alpha}$ for $\alpha\in(0,1)$ and $\overline{\Sigma}=\partial \Omega\cap \overline{\mathcal{C}}$ be an embedded hypersurface with boundary.  Then \eqref{eqn:1.1} and \eqref{Overdeterminedcondition} admit a weak solution $u\in C^{1}(\overline{\mathcal{C}}\backslash \Omega)$ if and only if $u$ is given by (\ref{model}) for the radius $R=\dfrac{n-p}{p-1}C^{-1}$ and $\overline{\Sigma} = \partial B_R(O) \cap \overline{\mathcal{C}}$.
\end{corollary}

For the proof of this Corollary, it is only necessary to note that we can obtain the orthogonal intersection in Theorem \ref{thm:1.2} by using the regularity assumption of solution $u$. 

In fact, due to $u\in  C^{1}(\overline{\mathcal{C}}\backslash \Omega)$ and the overdetermined condition (\ref{Overdeterminedcondition}), we know that $u=1$ and $|\nabla u|=C$ on $\overline{\Sigma}$. Then by the condition $\langle \nabla u,\nu_{\Gamma}\rangle=0$ on $\Gamma$, we have the restriction of the unit outward normal vector $\nu_{\partial \Omega}|_{\overline{\Sigma}}=-\dfrac{\nabla u}{|\nabla u|}$ satisfying $$\langle \nu_{\partial \Omega}|_{\overline{\Sigma}}, \nu_{\overline{\Gamma}} \rangle=0 \text{ on } \partial \overline{\Sigma}.$$ 

This means that $\overline{\Sigma}$ is perpendicular to $\partial \overline{\mathcal{C}}$ along $\partial{\overline{\Sigma}}.$ 

Furthermore, if we assume that Corollary \ref{coro:1.3} holds, we can also obtain Theorem \ref{thm:1.2} from Theorem \ref{regularity} below.

\quad

Here it is worth pointing out that the generalization presented in this paper is nontrivial. 
While this paper can be viewed as a parallel version of the paper \cite{CL} where $p=n$, our approach to establishing the uniqueness, regularity, and asymptotic behavior of the solution differs from theirs. Remarkably, the methods employed to prove rigidity results also exhibit differences. Specifically, due to the $n$-Laplace equation, without using the $P$-function and discussing the curvature in their paper, they can apply the isoperimetric inequality in the convex cone with the asymptotic behavior to directly get the rigidity result. However, in our paper, we only use the inequality case of the isoperimetric inequality and not the rigidity results that hold in the case of equality. The role of the isoperimetric inequality in this paper is to combine with the $P$-function to help us prove that the mean curvature of $\overline{\Sigma}$ has a positive lower bound that is related to some geometric quantities. As for the proof of the rigidity results, the main tools we use are the  Heintze-Karcher type inequality.
Furthermore,  since we focus on the convex cone case, previous conclusions drawn for Euclidean space are no longer applicable. Consequently, we must re-establish the existence, regularity, and asymptotic behavior of the solution $u$, as well as the properties of the $P$-function, etc. Additionally, even in the Laplace case $(p=2)$, our results are new. 

Another point to note is that Theorem 1.2 $(p=n)$ in the paper \cite{CL} and the theorems in the paper \cite{PT} include either the case $\mathcal{C} = \tilde{\mathcal{C}}\times\mathbb{R}^{k}$, where $\tilde{\mathcal{C}}\subset\mathbb{R}^{n-k}$ is an open convex cone with a vertex at the origin which contains no lines and $k \in\{1,\ldots, n\}$, or the case the centre of the spherical cap $x_{0}\in \partial \overline{\mathcal{C}}\backslash {O}$ and $\Sigma$ is a half-sphere lying over a flat portion of $\partial \overline{\mathcal{C}}$.
This immediately prompts the question of whether the results of this paper can be extended to the above cases. However, for the latter case, it seems impossible. In fact, the condition $\langle\nabla u(x), \nu_{\Gamma}\rangle=0  \text{ on } \Gamma$ does not hold.

\quad

The paper is organized as follows. In Section \ref{sec2}, we provide a brief overview of the notations, fundamental tools, the isoperimetric inequality in the convex cone, and the Heintze-Karcher type inequality in the convex cone. In Section \ref{sec3}, we establish the existence, regularity, and asymptotic behavior of the solution $u$ for the problem (\ref{eqn:1.1}).
In Section \ref{sec4}, we determine the value of constant $C$ in the overdetermined condition (\ref{Overdeterminedcondition}). By applying the maximum principle to a $P$-function, the isoperimetric inequality in the convex cone, and the value of $C$, we can establish a lower bound on the mean curvature of $\overline{\Sigma}$. Finally, we utilize the Heintze-Karcher type inequality to prove Theorem \ref{thm:1.2}.

\section{Preliminaries}\label{sec2}
\subsection{Divergence Theorem and comparison theorem}

In the current framework, we need the following generalized version of the divergence theorem and the comparison theorem. Throughout this paper, $\nu$ represents the unit outward normal vector, unless otherwise stated.
\begin{lemma}[Lemma 4.3 in \cite{CL}]\label{div}
     Let $E$ be a bounded open subset of $\mathbb{R}^{n}$ with Lipschitz boundary and let $f \in L^{1}(E)$. Assume that $\mathbf{a} \in C^{0}\left(\bar{E} ; \mathbb{R}^{n}\right)$ satisfies $\operatorname{div} \mathbf{a}=f$ in the sense of distributions in $E$. Then we have
\begin{equation}
    \int_{\partial E}\langle\mathbf{a}, \nu\rangle~ d \mathcal{H}^{n-1}=\int_{E} f(x)~ d x.
\end{equation}

\end{lemma} 

Although Lemma 2.4 in \cite{CL} is about anisotropic $n$-Laplace, we can use the same method to prove the following result.

\begin{lemma}\label{comparison}
     Let $\mathcal{C} \subset \mathbb{R}^{n}$ be an open convex cone, $E \subset \mathbb{R}^{n}$ be a bounded domain and $\Gamma_{0}:=\mathcal{C}\cap \partial E$ such that $\mathcal{H}^{n-1}\left(\Gamma_{0}\right)>0$ and $\mathcal{C} \cap E$ is connected. Assume that $p>1$, $u, v \in W^{1, p}(\mathcal{C} \cap E) \cap C^{0}\left((\mathcal{C} \cap E) \cup \Gamma_{0}\right)$ satisfy
$$
\begin{cases}-\Delta_{p} u \leq-\Delta_{p} v & \text { in } \mathcal{C} \cap E, \\ u \leq v & \text { on } \Gamma_{0}, \\ \langle \nabla u, \nu\rangle=\langle \nabla v, \nu\rangle=0 & \text { on } \Gamma_{1}:=\partial \overline{\mathcal{C}}\cap E.
\end{cases}
$$
Then $u \leq v$ in $\mathcal{C} \cap E$.
\end{lemma}

\subsection{Isoperimetric inequality and Heintze-Karcher type inequality}

Given an open subset $D \subset \mathbb{R}^{n}$ and a measurable set $E \subset \mathbb{R}^{n}$, we recall the definition of relative perimeter of $E$ in $D$, given by
$$
P(E ; D) =\sup \left\{\int_{E} \operatorname{div} \Phi ~d x: \Phi \in C_{0}^{1}\left(D ; \mathbb{R}^{n}\right), |\Phi| \leq 1\right\} .
$$
When the set $E$ has locally finite
perimeter,then $$P(E ; D) =\int_{D \cap \partial^{*} E} 1 ~d \mathcal{H}^{n-1}.$$
where $\partial^{*} E$ is the reduced boundary of $E$. 

The following isoperimetric inequality can be found in \cite[Theorem 1.3]{CRS}, \cite[Theorem 2.2]{FI} and \cite[Theorem 1.1]{LP}.

\begin{theorem}[Isoperimetric inequality in the convex cone]\label{isoperimetric}
For each measurable set $E \subset \mathbb{R}^{n}$ with $0<\mathcal{H}^{n}(\mathcal{C} \cap E)<\infty$, the following inequality holds:
$$
\frac{P(E ; \mathcal{C})}{\mathcal{H}^{n}(\mathcal{C} \cap E)^{\frac{n-1}{n}}} \geq \frac{P\left(B_{1} ; \mathcal{C}\right)}{\mathcal{H}^{n}\left(\mathcal{C} \cap B_{1}\right)^{\frac{n-1}{n}}}.
$$
Moreover, the equality holds if and only if $\mathcal{C}\cap E=\mathcal{C} \cap B_{R}\left(O\right)$ for some $R>0$.
\end{theorem}

\quad

From the equation (\ref{p-harmonic})
and 
$$\Delta u=-H|\nabla u|+\dfrac{u_{ij}u_{i}u_{j}}{|\nabla u|^{2}},$$ we know
\begin{equation}\label{mean curvature}
    |\nabla u|^{3} H=(p-1)u_{ij}u_{i}u_{j},
\end{equation}
where the mean curvature $H$ is with respect to the unit normal $\nu=-\dfrac{\nabla u}{|\nabla u|}$ pointing the exterior of $\Omega$.

We introduce a geometric inequality from the proof of \cite[Theorem 6.3]{PT}. This inequality will be important for obtaining the rigidity result.
\begin{lemma}[Heintze-Karcher type inequality in the convex cone]\label{HK}
Let $\overline{\Sigma}\subset \overline{\mathcal{C}}$ be an embedded $C^{2,\alpha}$ hypersurface with the boundary $\partial{\overline{\Sigma}}\subset \partial \overline{\mathcal{C}}\backslash \{O\}$ such that $\overline{\Sigma}$ and $\partial \overline{\mathcal{C}}$ along $\partial{\overline{\Sigma}}$ intersect orthogonally. Moreover, the closure of the region bounded by the hypersurface $\overline{\Sigma}$ and the boundary $\partial \overline{\mathcal{C}}$ contains the origin $O$.  Let $H_{\overline{\Sigma}}$ be the mean curvature of $\overline{\Sigma}$ with respect to $\nu_{\overline{\Sigma}}$.
If $H_{\overline{\Sigma}}>0$, then $$\int_{\overline{\Sigma}} \dfrac{n-1}{H_{\overline{\Sigma}}}~ d \mathcal{H}^{n-1}\geq n\mathcal{H}^{n}(\Omega\cap{\mathcal{C}})$$ and equality holds if and only if $\overline{\Sigma} = \partial B_R(O) \cap \overline{\mathcal{C}}$ for some $R>0$.
\end{lemma}

\section{Existence, Regularity and Asymptotic behavior}\label{sec3}

We first obtain the existence and uniqueness of the problem (\ref{eqn:1.1}). The proof outlined below is based on the ideas from \cite{CSZ, FM}.
\begin{lemma}\label{existence}
    Let $\Omega\subset \mathbb{R}^{n}$ be a bounded domain with a boundary of class $C^{2,\alpha}$. There exists a unique solution $u\in W^{1,p}_{loc}(\overline{{\mathcal{C}}}\backslash \Omega)$ to the problem (\ref{eqn:1.1}) in the distribution sense.  Moreover, $0< u\leq 1$
and
    \begin{equation}\label{eq:cap}
        {\rm Cap}_{p}(\Omega\cap{\mathcal{C}}; \mathcal{C})=\frac{1}{p}\int_{{\Omega}^{\mathcal{C}}}|\nabla u|^{p} ~dx.
    \end{equation} 
\end{lemma}

\begin{proof}
 Recall $${\rm Cap}_{p}(\Omega\cap{\mathcal{C}};\mathcal{C}):=\inf_{\phi}\left\{\frac{1}{p}\int_{{\mathcal{C}}}|\nabla\phi|^{p} ~dx: \phi\in \mathcal{A}\right\},$$
    where $$\mathcal{A}:=\left\{\phi\in W^{1,p}(\mathcal{C}): \phi-f=\omega \chi_{ {\mathcal{C}}} \text{ for }\omega\in W^{1,p}_{0}(\mathbb{R}^{n}\backslash \overline{\Omega}) \right\}.$$ 
   Here the function $f\in C^{\infty}_{c}(\mathbb{R}^{n})$, $0\leq f\leq 1$ and $f=1$ in a neighborhood of $\Omega.$

\textbf{Step 1:} We get a local version of the results.

 We consider 
$${\rm Cap}_{p}(\Omega\cap{\mathcal{C}};B_{R} \cap \mathcal{C}):=\inf_{\phi}\left\{\frac{1}{p}\int_{{B_{R} \cap\mathcal{C}}}|\nabla\phi|^{p} ~dx: \phi\in \mathcal{A}_{R}\right\},$$
    where $$\mathcal{A}_{R}:=\left\{\phi\in W^{1,p}(B_{R} \cap\mathcal{C}):  \phi-f=\omega^{R} \chi_{ {\mathcal{C}}} \text{ for }\omega^{R}\in W^{1,p}_{0}(B_{R} \backslash \overline{\Omega}) \right\}$$ for large $R$ such that the support set of $f$ is contained in $B_{R}$.

Then we choose a minimizing sequence $\{u_{j}\}\in\mathcal{A}_{R}$ such that $$\frac{1}{p}\int_{B_{R} \cap{\mathcal{C}}}|\nabla u_{j}|^{p} dx\rightarrow {\rm Cap}_{p}(\Omega\cap{\mathcal{C}};B_{R} \cap\mathcal{C}).$$ 

We set $$u_{j}:=f+\omega^{R}_{j}\chi_{ {\mathcal{C}}}.$$
It's easy to see that $\omega^{R}_{j}\chi_{ {\mathcal{C}}}\in W^{1,p}({B_{R} \cap\mathcal{C}}).$ Then by $ \omega_{j}^{R}\in W^{1,p}_{0}(B_{R} \backslash \overline{\Omega})$ and the Poincar\'{e} inequality in \cite[Corollary 4.5.3.]{ZW}, there exists a constant $C(R)$ depending on $R$ such that$$\int_{B_{R} \cap\mathcal{C}}|\omega^{R}_{j}\chi_{\mathcal{C}}|^{p} ~dx\leq C(R)\int_{B_{R} \cap\mathcal{C}}|\nabla \omega^{R}_{j}\chi_{\mathcal{C}}|^{p} ~dx$$
$$\leq C(R)(\int_{B_{R} \cap\mathcal{C}}|\nabla u_{j}|^{p}~ dx+\int_{B_{R} \cap\mathcal{C}} |\nabla f|^{p} ~dx) ,$$
which implies that $||u_{j}||_{W^{1,p}(B_{R} \cap\mathcal{C})}$ is bounded. So, there exists a subsequence (which we still denote as $\{u_j\}$ for convenience) such that $u_{j}\rightharpoonup u_{R} \in W^{1,p}(B_{R} \cap\mathcal{C})$ in the norm  $W^{1,p}$ as $j\rightarrow+\infty.$ Then the weak convergence implies $$\frac{1}{p}\int_{B_{R} \cap\mathcal{C}}|\nabla u_{R}|^{p} ~dx \leq \frac{1}{p}\lim\limits_{j\rightarrow +\infty}\int_{B_{R} \cap\mathcal{C}}|\nabla u_{j}|^{p}~ dx={\rm Cap}_{p}(\Omega\cap{\mathcal{C}};B_{R} \cap\mathcal{C}).$$

On the other hand, by $u_{j}=f+\omega^{R}_{j}\chi_{{\mathcal{C}}}$ and $||\omega^{R}_{j}\chi_{{\mathcal{C}}}||_{W^{1,p}}\leq||u_{j}||_{W^{1,p}}+||f||_{W^{1,p}},$ the sequence $\omega^{R}_{j}\chi_{{\mathcal{C}}}$ is bounded in $ W^{1,p}({B_{R} \cap\Omega^{\mathcal{C}}}),$ 
 there exists a function $\omega^{R}\in W^{1,p}(B_{R} \cap\Omega^{\mathcal{C}})$ such that $\omega^{R}_{j}\chi_{ {\mathcal{C}}}\rightharpoonup \omega^{R}$ in the norm  $W^{1,p}.$
Moreover, due to $ \omega_{j}^{R}\in W^{1,p}_{0}(B_{R} \backslash \overline{\Omega})$, we also know  $\omega^{R}=0$ on $\Sigma\cup (\partial B_{R}\cap \mathcal{C})$. Then by applying the extension theorem to $\omega^{R}$, we can find $\overline{\omega}^{R}\in W^{1,p}_{0}(B_{R} \backslash \overline{\Omega})$ such that $\overline{\omega}^{R}\chi_{\mathcal{C}}=\omega^{R}.$

Hence, we get $$u_{j}=f+\omega^{R}_{j}\chi_{ {\mathcal{C}}}\stackrel{W^{1,p}}{\rightharpoonup}u_{R}=f+\overline{\omega}^{R} \chi_{ {\mathcal{C}}}.$$ 
Therefore, $u_{R}\in \mathcal{A}_{R}$ and $${\rm Cap}_{p}(\Omega\cap{\mathcal{C}};B_{R} \cap\mathcal{C})\leq \frac{1}{p}\int_{B_{R} \cap\mathcal{C}}|\nabla u_{R}|^{p}~ dx \leq {\rm Cap}_{p}(\Omega\cap{\mathcal{C}};B_{R} \cap\mathcal{C}).$$
Meanwhile, from the expression of $u_{R}$, it can be seen that
\begin{equation}\label{caploc}
    {\rm Cap}_{p}(\Omega\cap{\mathcal{C}}; B_{R} \cap\mathcal{C})= \frac{1}{p}\int_{B_{R} \cap\mathcal{C}}|\nabla u_{R}|^{p} ~dx =\frac{1}{p}\int_{B_{R} \cap(\mathcal{C}\backslash\overline{\Omega})}|\nabla u_{R}|^{p} ~dx.
\end{equation}
    
Next, we solve the local version of the problem (\ref{eqn:1.1}). For any function $h_{R}\in W^{1,p}(B_{R} \cap\Omega^{{\mathcal{C}}})$ satisfying $h_{R}=0$ on $\Sigma$ and with bounded support in $B_{R} \cap\Omega^{{\mathcal{C}}}$, we can extend $h_{R}$ to $\overline{h}_{R}$ with bounded support in $B_{R}$ such that $\overline{h}_{R}\chi_{\mathcal{C}}=h_{R}$ in $\mathcal{C}$. Then $\overline{h}_{R}\in W^{1,p}_{0}(B_{R} \backslash \overline{\Omega})$ and $$u_{R}+th_{R}=f+\omega^{R}\chi_{ { {\mathcal{C}}}}+th_{R}=f+(\omega^{R}+t\overline{h}_{R})\chi_{ { {\mathcal{C}}}}\in \mathcal{A}_{R}.$$
Thus, by direct calculations, we have $$ 0=\left.\dfrac{d}{dt} \right|_{t=0}\frac{1}{p}\int_{B_{R} \cap\mathcal{C}}|\nabla (u_{R}+th_{R})|^{p} ~dx=\int_{B_{R} \cap\Omega^{\mathcal{C}}}(|\nabla u_{R}|^{p-2}\nabla u_{R}\cdot \nabla h_{R})~ dx.$$

To summarize, we have shown that the local problem 
    \begin{equation}\label{eqn:1.1loc}
\left\{
\begin{aligned}
&\Delta_{p}u=0\ \ {\rm in}\ \ B_{R} \cap{\Omega}^{\mathcal{C}},\\
&u=1\ \ {\rm on} \ \ \Sigma,\\
&\langle\nabla u(x), \nu_{\Gamma}\rangle=0\ \ {\rm on} \ \ \Gamma,\\
&u(x)=0\ \ {\rm on}\ \ \Gamma_{R}:=\partial B_{R}\cap\mathcal{C}.
\end{aligned}\right.
\end{equation}
has a weak solution $u_{R}\in  W^{1,p}({B_{R} \cap\Omega^{\mathcal{C}}})$.  

By applying Lemma $\ref{comparison}$, we know if $r>s$, then $u_{r}(x)\geq u_{s}(x)$ for any $x\in B_{s}\cap\Omega^{\mathcal{C}}$ and $0\leq u_{R}\leq 1$. 

\textbf{Step 2:}  We find a weak solution to the problem (\ref{eqn:1.1}).

From the monotonicity and boundedness of the solution as mentioned above, we can define a function $u:=\lim\limits_{R\rightarrow+\infty} u_{R}(x)$ for $x\in\Omega^{\mathcal{C}}$.
Let us prove that the function $u$ is a unique weak solution to the problem $(\ref{eqn:1.1})$.

We claim that \begin{equation}\label{limcap}
    \lim\limits_{R\rightarrow +\infty}{\rm Cap}_{p}(\Omega\cap{\mathcal{C}};B_{R} \cap\mathcal{C})= {\rm Cap}_{p}(\Omega\cap{\mathcal{C}};\mathcal{C}).
\end{equation}
In fact, on the one hand, from the definitions above, it's easy to see that $\mathcal{A}_{R}\subset \mathcal{A}$ and 
\begin{equation}\label{capgeq}
    {\rm Cap}_{p}(\Omega\cap{\mathcal{C}}; B_{R} \cap\mathcal{C})\geq {\rm Cap}_{p}(\Omega\cap{\mathcal{C}};\mathcal{C}).
\end{equation}
On the other hand, for any function $v\in\mathcal{A}$ with the form $v=f+V\chi_{\mathcal{C}},$ where $V\in W^{1,p}_{0}(\mathbb{R}^{n}\backslash \overline{\Omega})$, there exists a sequence $\{V_{j}\}\in C^{\infty}_{c}(\mathbb{R}^{n}\backslash \overline{\Omega})$ that converges to $V$ in $W^{1,p}.$ 
Note that when $j$ is fixed, $V_{j}\in W^{1,p}_{0}(B_{R}\backslash\overline{\Omega})$ for sufficiently large $R$. Then, from $V_{j}\chi_{\mathcal{C}}+f \in \mathcal{A}_{R}\rightarrow v$ in $W^{1,p},$ we have for any $\epsilon>0$, $$\frac{1}{p}\int_{\mathcal{C}}(|\nabla v|^{p} +\epsilon) ~dx \geq \frac{1}{p}\int_{\mathcal{C}}|\nabla V_{j}\chi_{\mathcal{C}}+\nabla f|^{p}~ dx \geq {\rm Cap}_{p}(\Omega\cap{\mathcal{C}};B_{R} \cap\mathcal{C}).$$
Due to  $$\dfrac{1}{p}\int_{\mathcal{C}}|\nabla V_{j}\chi_{\mathcal{C}}+\nabla f|^{p} ~dx \geq \lim_{R\rightarrow +\infty}{\rm Cap}_{p}(\Omega\cap{\mathcal{C}};B_{R} \cap\mathcal{C})$$ for any $j$, we know that $$\frac{1}{p}\int_{\mathcal{C}}|\nabla v|^{p} ~dx = \lim_{j\rightarrow +\infty}\frac{1}{p}\int_{\mathcal{C}}|\nabla V_{j}\chi_{\mathcal{C}}+\nabla f|^{p} ~ dx\geq \lim_{R\rightarrow +\infty}{\rm Cap}_{p}(\Omega\cap{\mathcal{C}};B_{R} \cap\mathcal{C}).$$
 
Now thanks to the interior regularity results in \cite[Theorems 1 and 2]{DiB}, the identities (\ref{caploc}) and (\ref{limcap}), we can deduce that the family of solutions $u_{R}$ is bounded in $C^{1,\alpha}_{loc}(B_{R}\cap \Omega^{\mathcal{C}})\cap W^{1,p}_{loc}(\overline{B_{R}\cap \Omega^{\mathcal{C}}})$, uniformly for large $R$. Then by the Arzel\`{a}-Ascoli theorem and a diagonal process, we can find a sequence $R_{j}\rightarrow +\infty$ such that $u_{R_{j}}\rightarrow u$ in $C^{1,\alpha}_{loc}(\Omega^{\mathcal{C}})$ and $u_{R_{j}}\rightharpoonup u$ in $W^{1,p}_{loc}(\overline{\Omega^{\mathcal{C}}}).$ Then we know that $u$ is a weak solution to the problem 
\begin{equation*}
\left\{
\begin{aligned}
&\Delta_{p}u=0 \text{ in }  {\Omega}^{\mathcal{C}},\\
&u=1 \text{ on }  {\Sigma},\\
&\langle\nabla u(x), \nu_{\Gamma}\rangle=0  \text{ on } \Gamma.
\end{aligned}\right.
\end{equation*} 
Moreover, the inequality $u>0$ and the asymptotic behavior $u(x)\rightarrow 0$ as $|x|\rightarrow+\infty$ can be proven by applying $u=\lim\limits_{R\rightarrow+\infty} u_{R}(x)$ for $x\in\Omega^{\mathcal{C}}$ and the comparison theorem to $u_{R}$ and the fundamental solutions. The inequality $u\leq 1$ and the uniqueness of the solution can be directly obtained from the comparison theorem. 

\textbf{Step 3:} We will show  the following expression of the capacity
$${\rm Cap}_{p}(\Omega\cap{\mathcal{C}};\mathcal{C})= \frac{1}{p}\int_{\mathcal{C}\backslash\overline{\Omega}}|\nabla u|^{p} ~dx.$$

From the Fatou's Lemma, 
\begin{equation}\label{eq:cap_p>}
\begin{aligned}
&{\rm Cap}_{p}(\Omega\cap{\mathcal{C}};\mathcal{C})=\lim_{R\rightarrow +\infty}{\rm Cap}_{p}(\Omega\cap{\mathcal{C}};B_{R} \cap\mathcal{C})\\
&=\lim_{R\rightarrow+\infty}\frac{1}{p}\int_{B_{R} \cap\Omega^{\mathcal{C}}}|\nabla u_{R}|^{p} ~dx\geq \frac{1}{p}\int_{\Omega^{\mathcal{C}}}|\nabla u|^{p}~ dx.
\end{aligned}
\end{equation}

Furthermore, since the sequence $\nabla u_{R}$ is bounded in $L^{p}(\Omega^{\mathcal{C}})$ uniformly in large $R$ and $u_{R_{j}}\rightarrow u$ in $C^{1,\alpha}_{loc}(\Omega^{\mathcal{C}})$, we know $\nabla u_{R_{j}} \rightharpoonup \nabla u$ in the norm  $L^{p}(\Omega^{\mathcal{C}}).$ Then by applying Mazur lemma, we know that for any $\epsilon>0, \exists \lambda_{i}\geq 0$ $(i=1,2, \ldots, N)$ with $\sum\limits^{N}_{i=1} \lambda_{i}=1$, such that $$||\nabla u- \sum^{N}_{i=1}\lambda_{i}\nabla u_{R_{i}}||_{L^{p}}<\epsilon.$$

Then, by using the fact that $u_{R_{i}}\in \mathcal{A}_{R_{i}}$ and the inequality (\ref{capgeq}), we can conclude 
\begin{equation}\label{eq:cap_p<}
\frac{1}{p}\int_{\Omega^{\mathcal{C}}}|\nabla u|^{p}~ dx \geq {\rm Cap}_{p}(\Omega\cap{\mathcal{C}};\mathcal{C}).
\end{equation}
Thus the claim follows from \eqref{eq:cap_p>} and \eqref{eq:cap_p<} and 
 this completes the proof of Lemma \ref{existence}.
\end{proof}

To get a higher regularity result, we need the assumption of orthogonal intersection.
\begin{theorem}\label{regularity}
    Suppose that $\overline{\Sigma}$ is perpendicular to $\partial \overline{\mathcal{C}}$ along $\partial{\overline{\Sigma}}$. 
  Let $u$ be a solution of (\ref{eqn:1.1}).
    Then the solution $u\in C^{1}(\overline{\mathcal{C}}\backslash \Omega).$ 
\end{theorem}
\begin{proof}
First, we can apply the regularity results in \cite{DiB} and \cite[Theorems 1 and 2]{Lie} to obtain $u\in C^{1,\alpha}(\Omega^{\mathcal{C}}\cup \Sigma\cup \Gamma).$ Next, we use the technique of flattening the boundary of the barrier and planar reflection as in \cite{PT} to prove the smoothness up to $\partial{\overline{\Sigma}}$. So, we fix a point $q_{0} \in \partial{\overline{\Sigma}}$. 

Following the construction in \cite[Theorem 6.1]{PT}, in a neighborhood $\mathcal{U}$ of $q_{0}$, we can set
$$
\mathcal{C} \cap \mathcal{U}=\left\{q=\left(q^{\prime}, q_{n}\right) \in \mathcal{U}: q_{n}>g\left(q^{\prime}\right)\right\}$$ and   $$\partial \overline{\mathcal{C}} \cap \mathcal{U}=\left\{q=\left(q^{\prime}, q_{n}\right) \in \mathcal{U}: q_{n}=g\left(q^{\prime}\right)\right\}
$$
for some smooth function $g$. 
Denoting $q_{0}=\left(q_{0}^{\prime}, g\left(q_{0}^{\prime}\right)\right)$,  $\nabla^{\prime}$ the gradient in the $q^{\prime}$-variables and $\nabla'^{2}$ the Hessian in the $q^{\prime}$-variables, we define a map $\psi: \mathcal{U} \rightarrow \mathbb{R}^{n}$ as
$$
\psi(q)=\psi\left(q^{\prime}, q_{n}\right)=\left(q^{\prime}-q_{0}^{\prime}-\frac{g\left(q^{\prime}\right)-q_{n}}{1+\left|\nabla^{\prime} g\left(q^{\prime}\right)\right|^{2}} \nabla^{\prime} g\left(q^{\prime}\right), q_{n}-g\left(q^{\prime}\right)\right) .
$$

At any point $q \in  \mathcal{U}$, by direct computations, we can get the Jacobian $\mathcal{J}\psi$ of $\psi$, where the matrix components are as follows: 
$$\psi^{k}_{l}=\delta^{k}_{l}-\frac{ g_{k}\left(q^{\prime}\right)  g_{l}\left(q^{\prime}\right)}{1+\left|\nabla^{\prime} g\left(q^{\prime}\right)\right|^{2}}+(q_{n}-g(q')) \cdot$$
$$(\frac{g_{kl}(q')}{1+|\nabla' g(q')|^{2}}-\frac{2(\nabla'^{2}g(q')\nabla'g(q'))_{l} g_{k}(q')}{(1+|\nabla' g(q')|^{2})^{2}}),$$
$$\psi^{k}_{n}=\frac{g_{k}\left(q^{\prime}\right)}{1+\left|\nabla^{\prime} g\left(q^{\prime}\right)\right|^{2}}, \quad \psi^{n}_{l}=-g_{l}\left(q^{\prime}\right)\text{ and }\psi^{n}_{n}=1,$$
where $g_{k}(q')$ $(k=1,\ldots, n-1)$ and $g_{k l}(q')$ $(k,l=1,\ldots, n-1)$ are the components of the vector $\nabla' g(q')$ and the matrix $\nabla'^{2}g(q').$

Moreover, it is easy to see that the Jacobian $\mathcal{J}\psi$ is an invertible matrix on  $\partial\overline{\mathcal{C}}\cap \mathcal{U}$, hence we can find a connected neighbourhood $\mathcal{U}_{0} \subseteq \mathcal{U}$ of $q_{0}\in \partial \overline{\mathcal{C}}\cap \mathcal{U}$ such that $\psi|_{\mathcal{U}_0}$ is a diffeomorphism. Let $z:=\psi(q)$ for $q \in  \mathcal{U}_{0}$ and $\tilde{f}:=f \circ \psi^{-1}$ for any smooth function $f$ defined on $\mathcal{U}_0$, then
$$
\frac{-1}{\sqrt{1+\left|\nabla^{\prime} g\left(q^{\prime}\right)\right|^{2}}} \frac{\partial f}{\partial \nu}(q)=\frac{\partial \tilde{f}}{\partial z_{n}}(z) \text { at } q \in \partial \overline{\mathcal{C}} \cap \mathcal{U}_{0} \Longleftrightarrow z \in \psi\left(\mathcal{U}_{0}\right) \cap\left\{z_{n}=0\right\}.
$$

Let us denote
$$
\tilde{\Omega}^{\mathcal{C}}=\psi\left(\Omega^{\mathcal{C}} \cap \mathcal{U}_{0}\right),  M_{0}=\psi\left(\Sigma
\cap \mathcal{U}_{0}\right),  M_{1}=\psi\left(\Gamma \cap \mathcal{U}_{0}\right),  v(z)=u\left(\psi^{-1}(z)\right).
$$
Let $v_{i}$ be the components of the vector $\nabla v$ and $\varphi$ be the inverse matrix of $\mathcal{J}\psi$.

From the equation in the problem (\ref{eqn:1.1}) we get
$$
\begin{aligned}
&0=\divg(|\nabla u|^{p-2}\nabla u)(q)=(|\nabla u|^{p-2}\psi^{i}_{k}v_{i})_k\\
&=|\nabla u|^{p-2}[(\psi^{i}_{k})_{k}v_{i}+\psi_{k}^{i}\psi_{k}^{j}v_{i j}]+(p-2)|\nabla u|^{p-4}\cdot\\
&[v_{l j}\psi_{q}^{l}\psi_{q}^{m}v_{m}\psi^{j}_{k}\psi_{k}^{i}v_{i}+v_{l}(\psi_{q}^{l})_{k}\psi_{q}^{m}v_{m}\psi_{k}^{i}v_{i}].
\end{aligned}
$$

On the other hand, by direct computations, we obtain
$$
\begin{aligned}
&\divg(|\nabla u|^{p-2}(\mathcal{J}\psi)^{T} \mathcal{J}\psi \nabla v)(z)=(|\nabla u|^{p-2}\psi^{i}_{k}\psi_{k}^{j}v_{j})_{i}\\
&=|\nabla u|^{p-2}(\psi^{i}_{k}\psi_{k}^{j}v_{i j}+(\psi^{i}_{k})_{m}\varphi^{m}_{i}\psi_{k}^{j}v_{j}+(\psi_{k}^{j})_{k}v_{j})\\
&+(p-2)|\nabla u|^{p-4}[v_{i l}\psi_{q}^{l}\psi_{q}^{m}v_{m}\psi^{i}_{k}\psi_{k}^{j}v_{j}+v_{l}(\psi_{q}^{l})_{k}\psi_{q}^{m}v_{m}\psi_{k}^{j}v_{j}].
\end{aligned}
$$
Thus,
\begin{equation*}
\begin{aligned}
&\divg(|\nabla u|^{p-2}(\mathcal{J}\psi)^{T}  \mathcal{J}\psi \nabla v)(z)-|\nabla u|^{p-2}(\psi^{i}_{k})_{m}\varphi^{m}_{i}\psi_{k}^{j}v_{j}=0.
\end{aligned}
\end{equation*}
Consequently,
\begin{equation}\label{regmaineq}
\begin{cases}\divg A(z, \nabla v)+\langle B(z, \nabla v),\nabla v\rangle=0 & \text { in } \tilde{\Omega}^{\mathcal{C}}, \\ v=1 & \text { on } M_{0} ,\\ \dfrac{\partial v}{\partial z_{n}}=0 & \text { on } M_{1},\end{cases}
\end{equation}
where $$A:\mathbb{R}^{n}\times \mathbb{R}^{n}\rightarrow \mathbb{R}^{n}$$ $$(z, \eta)\rightarrow|\mathcal{J} \psi ~\eta|^{p-2}(\mathcal{J}\psi)^{T}  \mathcal{J}\psi ~\eta$$ and the components of vector $B$ $$B^{j}:\mathbb{R}^{n}\times \mathbb{R}^{n}\rightarrow \mathbb{R}$$ $$(z, \eta)\rightarrow -|\mathcal{J} \psi ~\eta|^{p-2}(\psi^{i}_{k})_{m}\varphi^{m}_{i}\psi_{k}^{j}.$$ 

Moreover, if we denote $$a^{i j}(z):=\frac{\partial A^{i}}{\partial \eta_{j}}=|\mathcal{J}\psi ~\eta|^{p-2}\psi_{k}^{i}\psi^{j}_{k}
+(p-2)|\mathcal{J}\psi ~\eta|^{p-4}\psi_{q}^{j}\psi_{q}^{m}\eta_{m}\psi^{i}_{k}\psi_{k}^{l}\eta_{l}$$ and
$$b^{i}(z):=|\mathcal{J} \psi ~\eta|^{p-2}(\psi_{k}^{i})_{k}+(p-2)|\mathcal{J} \psi ~\eta|^{p-4}v_{l}(\psi_{q}^{l})_{k}\psi_{q}^{m}v_{m}\psi_{k}^{i},$$
then it's easy to know that the equation in (\ref{regmaineq}) can be rewritten as 
\begin{equation}\label{mideq}
a^{i j}(z)v_{i j}(z)+b^{i}(z)v_{i}(z)=0,
\end{equation} 
and it is a degenerate quasilinear elliptic equation.

Let us define $$
\begin{aligned} 
\omega(z)= \begin{cases}v(z) & \text { if } z \in \tilde{\Omega}^{\mathcal{C}} \cup M_{1}, \\
v\left(z^{\prime},-z_{n}\right) & \text { if }\left(z^{\prime},-z_{n}\right) \in \tilde{\Omega}^{\mathcal{C}} ,\end{cases}
\end{aligned}
$$
and 
$$
\begin{aligned} 
\psi^{i}_{k}\psi^{j}_{k}=&
\begin{cases}
\psi^{i}_{k}\psi^{j}_{k} &\text { if } z \in \tilde{\Omega}^{\mathcal{C}} \cup M_{1}, \\
\psi^{i}_{k}\psi^{j}_{k} &\text { if }\left(z^{\prime},-z_{n}\right) \in \tilde{\Omega}^{\mathcal{C}}\\
\end{cases} \text{with } 1\leq i, j \leq n-1 \text { or }(i, j)=(n, n), \\&
\begin{cases}
\psi^{i}_{k}\psi^{j}_{k} & \text {if } z \in \tilde{\Omega}^{\mathcal{C}} \cup M_{1}, \\
-\psi^{i}_{k}\psi^{j}_{k} & \text {if }\left(z^{\prime},-z_{n}\right) \in \tilde{\Omega}^{\mathcal{C}}
\end{cases} \text {with } i\text { or } j=n \text{ and } (i,j)\neq (n,n), \\
 B^{j}(z)=&\left\{\begin{array}{ll}
B^{j}(z) & \text { if } z \in \tilde{\Omega}^{\mathcal{C}} \cup M_{1}, \\
B^{j}\left(z^{\prime},-z_{n}\right) & \text { if }\left(z^{\prime},-z_{n}\right) \in \tilde{\Omega}^{\mathcal{C}}
\end{array} \text { with } 1\leq j \leq n-1,\right.\\
&\left\{\begin{array}{ll} B^{j}(z) & \text { if } z \in \tilde{\Omega}^{\mathcal{C}}\cup M_{1}, \\
-B^{j}\left(z^{\prime},-z_{n}\right) & \text { if }\left(z^{\prime},-z_{n}\right) \in \tilde{\Omega}^{\mathcal{C}} \end{array} \quad \text { with } j=n , \right.\\
 b^{j}(z)=&\left\{\begin{array}{ll}
b^{j}(z) & \text { if } z \in \tilde{\Omega}^{\mathcal{C}} \cup M_{1}, \\
b^{j}\left(z^{\prime},-z_{n}\right) & \text { if }\left(z^{\prime},-z_{n}\right) \in \tilde{\Omega}^{\mathcal{C}}
\end{array} \text { with } 1\leq j \leq n-1,\right.\\
&\left\{\begin{array}{ll} b^{j}(z) & \text { if } z \in \tilde{\Omega}^{\mathcal{C}}\cup M_{1}, \\
-b^{j}\left(z^{\prime},-z_{n}\right) & \text { if }\left(z^{\prime},-z_{n}\right) \in \tilde{\Omega}^{\mathcal{C}}  \end{array} \quad \text { with } j=n. \right. 
\end{aligned}
$$
Now we can check 
$$
\begin{aligned}
 A^{j}(z,\nabla \omega)=&\left\{\begin{array}{ll}
A^{j}(z) & \text { if } z \in \tilde{\Omega}^{\mathcal{C}} \cup M_{1}, \\
A^{j}\left(z^{\prime},-z_{n}\right) & \text { if }\left(z^{\prime},-z_{n}\right) \in \tilde{\Omega}^{\mathcal{C}}
\end{array} \text { with } 1\leq j \leq n-1,\right.\\
&\left\{\begin{array}{ll} A^{j}(z) & \text { if } z \in \tilde{\Omega}^{\mathcal{C}}\cup M_{1}, \\
-A^{j}\left(z^{\prime},-z_{n}\right) & \text { if }\left(z^{\prime},-z_{n}\right) \in \tilde{\Omega}^{\mathcal{C}}  \end{array} \quad \text { with } j=n, \right.
\end{aligned}
$$
$a^{i j}(z)=$
$$
\begin{aligned}
&\begin{cases}
a^{i j}(z) & \text {if } z \in \tilde{\Omega}^{\mathcal{C}} \cup M_{1}, \\
a^{i j}\left(z^{\prime},-z_{n}\right) & \text {if }\left(z^{\prime},-z_{n}\right) \in \tilde{\Omega}^{\mathcal{C}}\\
\end{cases} \text {with } 1\leq i, j \leq n-1 \text { or }(i, j)=(n, n), \\
&\begin{cases}
a^{i j}(z) & \text { if } z \in \tilde{\Omega}^{\mathcal{C}} \cup M_{1}, \\
-a^{i j}\left(z^{\prime},-z_{n}\right) & \text { if }\left(z^{\prime},-z_{n}\right) \in \tilde{\Omega}^{\mathcal{C}}
\end{cases} \text {with } i \text { or } j=n \text{ and } (i,j)\neq (n,n),  \\
\end{aligned}
$$
and the equation (\ref{mideq}) can be extended to $a^{i j}(z)\omega_{i j}(z)+b^{i}(z)\omega_{i}(z)=0$.

Furthermore, the following inequalities hold:
$$
\begin{aligned}
 a^{i j}&\xi_{i}\xi_{j}\geq \lambda|\eta|^{p-2}|\xi|^{2},\\
&|a^{i j}|\leq \Lambda|p|^{p-2},\\
|A(z, \eta)&-A(z_{1}, \eta)|\leq \Lambda(1+|\eta|)^{p-1}|z-z_{1}|^{\alpha},\\
|B(z,\eta)|\leq\Lambda(1+ |\eta|)^{p} \text{ for all }&(z, \eta)\in\partial\Omega_{ref}^{\mathcal{C}}\times \mathbb{R}^{n} \text{ and all }z_{1}\in \Omega_{ref}^{\mathcal{C}},\xi\in\mathbb{R}^{n},
\end{aligned}
$$
where $$\Omega^{\mathcal{C}}_{r e f}=\left\{z=\left(z^{\prime}, z_{n}\right): z \in \tilde{\Omega}^{\mathcal{C}} \cup M_{1} \text { or }\left(z^{\prime},-z_{n}\right) \in \tilde{\Omega}^{\mathcal{C}}\right\}$$ and constants $\lambda\leq \Lambda.$

Consequently, the boundary $\partial \Omega^{\mathcal{C}}_{r e f}$ is $C^{2,\alpha}$-smooth and the function $\omega$ is a solution to the divergence structure equation
$$
    \divg A(z, \nabla \omega)+\langle B(z, \nabla \omega),\nabla \omega\rangle=0  \text { in } {\Omega}^{\mathcal{C}}_{ref}, 
$$in the distribution sense.
Therefore, due to $0<u\leq 1$, we can deduce from \cite[Theorem 1]{Lie} that $\omega \in C^{1, \alpha}(\overline{{\Omega}^{\mathcal{C}}_{ref}})$.
\end{proof}

\begin{remark}
    1.When $\mathcal{C}=\mathbb{R}^{n}$, we can directly apply the classic regularity result in \cite{Lie} to obtain this regularity result. 
    
   2. For the case $p=2$, since the equation is strictly elliptic, we can also use the $L^{p}$-theory to obtain the result as in \cite{PT}.  

   3. As shown in the introduction, for the overdetermined problem (\ref{eqn:1.1}) and (\ref{Overdeterminedcondition}), the assumption of orthogonal intersection is equivalent to the assumption of regularity.
\end{remark}

\quad

For $1<p<n$, the fundamental solution to $\Delta_{p} u=0$ in $\mathbb{R}^{n}\backslash \{O\}$ is given by
$$\Gamma_{p}(x):=|x|^{\frac{p-n}{p-1}}.$$ It is easy to prove that $\Delta_{p} \Gamma_{p}=\delta_{0}$ in $\mathbb{R}^{n},$ where $\delta_{0}$ is the Dirac function about the origin.

Now, let us investigate the asymptotic behavior of the solution to the problem ($\ref{eqn:1.1}$). First, we introduce the following Hopf-type lemma from \cite[Proposition 3.3.1.]{Tolk} and derive a corresponding strong maximum principle.

\begin{lemma}\label{hopf}
Assume that the functions $u_{1}\in C^{1}(\overline{\mathcal{C}}\backslash{\Omega})$, $u_{2}\in C^{2}(\overline{\mathcal{C}}\backslash{\Omega})$ are solutions to the following problem 
 \begin{equation}\label{eqn:3.2}
\left\{
\begin{aligned}
&\Delta_{p}u=0\ \ {\rm in}\ \ \Omega^{\mathcal{C}},\\
&\langle\nabla u(x), \nu_{\Gamma}\rangle=0\ \ {\rm on} \ \ \Gamma.
\end{aligned}\right.
\end{equation}

Let $B\subset {\Omega}^{\mathcal{C}}$ be a ball and we further assume that the solutions $u_{1}, u_{2}$ satisfy 
\begin{equation}\label{hopfcondition}
\begin{aligned}
    & u_{1}>u_{2} \text{ in } B,\\
   &u_{1}(x_{0})=u_{2}(x_{0})\text{ for some }x_{0}\in\partial B.
\end{aligned}
\end{equation}
Additionally, we suppose that $u_{2}$ satisfies 
\begin{equation}\label{u2}
    |\nabla u_{2}|\geq \delta \text{ in } B,
\end{equation} for some positive constant $\delta$.

Then \begin{equation}\label{C1est}
    u_{1 \nu}(x_{0})\neq  u_{2 \nu}(x_{0}).
\end{equation}
\end{lemma}

\begin{remark}
First, note that the point $x_{0}$ can be on  $\Gamma\cup\Sigma$ and the function $u_{2}$ will be chosen as a constant multiple of the fundamental solution $\Gamma_{p}$ in later applications.

Next, instead of assuming \eqref{u2} in the entire $\Omega^{\mathcal{C}}$ as in \cite{Tolk}, we only require it to hold in $B$, ensuring that  $\Gamma_p$ satisfies this condition without changing the original proof. 
\end{remark}

\quad

With the assumption that ${\Omega}^{\mathcal{C}}$ is connected, we can get the following result: 

\begin{corollary}\label{smp}
Suppose that the functions $u_{1}\in C^{1}(\overline{\mathcal{C}}\backslash{\Omega})$, $u_{2}\in C^{2}(\overline{\mathcal{C}}\backslash{\Omega})$ satisfy (\ref{eqn:3.2}), the inequality in (\ref{u2}) for any ball $B\subset \overline{\mathcal{C}}\backslash \Omega$ and $$u_{1}\geq u_{2} \text{ and } u_{1} \not\equiv u_{2} \text{ in } {\Omega}^{\mathcal{C}}\cup \Gamma.$$
Then $u_{1}>u_{2}$ in ${\Omega}^{\mathcal{C}}\cup \Gamma.$
\end{corollary}

\begin{proof}
    We consider $$E:=\{x\in{\Omega}^{\mathcal{C}}\cup\Gamma: u_{1}(x)=u_{2}(x)\}.$$
From the assumption, we know $E\neq {\Omega}^{\mathcal{C}}\cup\Gamma.$ We claim that $E=\emptyset$.

Otherwise, for the case where $E\cap {\Omega}^{\mathcal{C}}\neq \emptyset$,  there exists a point $x_{1}\in{\Omega}^{\mathcal{C}}$ such that $$u_{1}(x_{1})=u_{2}(x_{1}).$$ 
Then, from the continuity of the solutions $u_{1}, u_{2},$ the set $$E':=\{x\in {\Omega}^{\mathcal{C}}:u_{1}(x)=u_{2}(x)\}\neq \emptyset$$ is a relatively closed subset in ${\Omega}^{\mathcal{C}}$.
    So we can choose a point $x_{0}\in E'$ satisfying the conditions (\ref{hopfcondition}) in Lemma \ref{hopf}. 
In fact, the set ${\Omega}^{\mathcal{C}}\backslash E'$ is a relatively
5open subset in ${\Omega}^{\mathcal{C}}$. There exists one point $x^{*}\in {\Omega}^{\mathcal{C}}\backslash E'$  such that $dist(x^{*}, \partial {\Omega}^{\mathcal{C}})>dist(x^{*},  E')$ with respect to the standard distance function of $\mathbb{R}^{n}$. We take a ball $B_{r}(x^{*})$ with $r=dist(x^{*}, E').$ Then we find the point $x_{0}\in \partial B_{r}(x^{*})\cap E'.$  This contradicts the conclusion \eqref{C1est} in Lemma \ref{hopf}.
    Thus, we obtain  $E\cap {\Omega}^{\mathcal{C}}=\emptyset$. 

For the case where $E\cap \Gamma\neq \emptyset$,  there exists a point $x_{2}\in \Gamma$ such that $u_{1}(x_{2})=u_{2}(x_{2}).$ Since the boundary $\Gamma$ is smooth, we can find an interior ball $B\subset{\Omega}^{\mathcal{C}}$ such that  $u_{1}>u_{2} \text{ in } B$ and
   $u_{1}(x_{2})=u_{2}(x_{2})\text{ with }x_{2}\in\partial B\cap \Gamma.$
So Lemma \ref{hopf} implies $$u_{1\nu_{\Gamma}}<u_{2\nu_{\Gamma}},$$ which contradicts the condition $\langle\nabla (u_{1}-u_{2}), \nu_{\Gamma}\rangle=0$ in (\ref{eqn:3.2}).
\end{proof}

\begin{remark}\label{boundeddomain}
    This corollary also holds for the connected subset of $\overline{\mathcal{C}}\backslash \Omega$.
\end{remark}

Now we use the technique introduced in \cite{KV} to obtain the asymptotic behavior of the solution in convex cone $\mathcal{C}$, which is an extension of \cite[Lemma 2.15.]{CSZ}.
\begin{theorem}[Asymptotic behavior]\label{Asymptotic expansion}

Let $\Omega\subset\mathbb{R}^n$ be a bounded domain with a boundary of class ${C}^{2,\alpha}$. Let $u$ be a
 solution of the problem (\ref{eqn:1.1}) in ${\Omega}^{\mathcal{C}}$. Then there exists a positive constant $\gamma$ such that $u$ satisfies
\begin{flalign}\label{asy1}
   &\ (\romannumeral1) \lim\limits_{|x|\rightarrow+\infty}\dfrac{u(x)}{\Gamma_{p}(x)}={\gamma},&
\end{flalign}
\begin{flalign}\label{asy2}
  &\  (\romannumeral2) \nabla u(x)=\gamma\nabla \Gamma_{p}(x)+o(|x|^{-\frac{n-1}{p-1}}), \text{ as } |x|\rightarrow +\infty.&
\end{flalign}

\end{theorem}

\begin{proof}

\textbf{Step 1:} There exist $C^{0}$ estimates of the function $\dfrac{u(x)}{\Gamma_{p}(x)}$.

We apply the comparison theorem to the following problems: for any $\epsilon>0$,
\begin{equation}\label{eqn:1.2}
\left\{
\begin{aligned}
&\Delta_{p}(u+\epsilon)=\Delta_{p} U_{1}=0\ \ {\rm in}\ \ {\Omega}^{\mathcal{C}},\\
&u+\epsilon\geq U_{1} \ \ {\rm on} \ \ \Sigma,\\
&u+\epsilon\geq U_{1}\ \ {\rm as}\ \ |x|\rightarrow+\infty,\\
&\langle\nabla (u+\epsilon), \nu_{\Gamma}\rangle=\langle\nabla U_{1}, \nu_{\Gamma}\rangle=0\ \ {\rm on} \ \ \Gamma,
\end{aligned}\right.
\end{equation}
and
 \begin{equation}\label{eqn:1.3}
\left\{
\begin{aligned}
&\Delta_{p}u=\Delta_{p} (U_{2}+\epsilon)=0\ \ {\rm in}\ \ {\Omega}^{\mathcal{C}},\\
&u\leq U_{2}+\epsilon\ \ {\rm on} \ \ \Sigma,\\
&u\leq U_{2}+\epsilon\ \ {\rm as}\ \ |x|\rightarrow+\infty,\\
&\langle\nabla u, \nu_{\Gamma}\rangle=\langle\nabla (U_{2}+\epsilon), \nu_{\Gamma}\rangle=0\ \ {\rm on} \ \ \Gamma,
\end{aligned}\right.
\end{equation}
where the functions $U_1, U_2: \overline{\mathcal{C}}\backslash \{O\} \rightarrow \mathbb{R}$ are defined by $U_{1}(x):=R_{1}^{\frac{n-p}{p-1}} \Gamma_{p}(x)$, $U_{2} (x):=R_{2}^{\frac{n-p}{p-1}} \Gamma_{p}(x)$   and the constants $0<R_{1}<R_{2}$ satisfy
$$
R_{1}:=\sup \{r>0: B_{r} \subset \Omega\}, \quad R_{2}:=\inf \{r>0: \Omega\subset B_{r}\}.
$$

Then letting $\epsilon\rightarrow0$, we obtain the following inequality:
\begin{equation}\label{3.4est1}
R_{1}^{\frac{n-p}{p-1}} \leq \dfrac{u(x)}{\Gamma_{p}(x)} \leq R_{2}^{\frac{n-p}{p-1}}
\end{equation}
for any $x\in{\overline{\mathcal{C}}}$ such that $|x| \geq R_{2}$.

\textbf{Step 2:} There exists a $C^{1,\beta}$ estimate of the solution $u$.

We define $$
V_{R_{0}}(y):=u(R_{0} y) R_{0}^{\frac{n-p}{p-1}},
$$ for any constant $R_{0}>4 R_{2}$ and $y\in D:=\left\{y \in {\overline{\mathcal{C}}}: \frac{1}{4}<|y|<4\right\}$. It can be shown that $V_{R_{0}}$ is a $p$-harmonic function in $D$.
Moreover, from (\ref{3.4est1}) we have that $V_{R_{0}}$ is bounded in $D$ by a constant depending on $n,p$ and $\Omega$. 

Therefore, we can apply Theorem $2$ in \cite{Lie} to obtain the estimates: 
\begin{equation*}
    |\nabla V_{R_{0}}(y)| \leq C_{1}, \quad\left|\nabla V_{R_{0}}(y)-\nabla V_{R_{0}}(y^{\prime})\right| \leq C_{2}\left|y-y^{\prime}\right|^{\beta},
\end{equation*}
where $y,y^{\prime} 
\in \Tilde{D}:=\{y\in \overline{\mathcal{C}}: \dfrac{1}{2} \leq|y| \leq 2\}$ and  the positive constants $C_{1},C_{2}$ and $\beta \in(0,1)$ are independent of $R_{0}$.

Hence, we have 
\begin{equation}\label{C1}
    R_{0}^{\frac{n-p}{p-1}} R_{0}|\nabla u(x)| \leq C_{1}
   \text{ and }
R_{0}^{\frac{n-p}{p-1}} R_{0}\left|\nabla u(x)-\nabla u(x^{\prime})\right| \leq C_{2} \frac{\left|x-x^{\prime}\right|^{\beta}}{R_{0}^{\beta}},
\end{equation}
$\text{ where }x, x^{\prime}\in\overline{\mathcal{C}} \text{ satisfy }\frac{1}{2} R_{0} \leq|x|,\left|x^{\prime}\right| \leq 2 R_{0}.$

Since the constant $R_{0}$ is arbitrary, the estimate (\ref{C1}) implies that there exists $\tilde{C}_{1}>0$ independent of $R_{0}$ such that
\begin{equation}\label{C1asy}
|x||\nabla u(x)| \leq \tilde{C}_{1} \Gamma_{p}(x)
\text{ and }
\left|\nabla u(x)-\nabla u(x^{\prime})\right| \leq \tilde{C}_{2} \frac{\Gamma_{p}(x)}{|x|^{\beta+1}}\left|x-x^{\prime}\right|^{\beta},
\end{equation}
where $x, x^{\prime}\in\overline{\mathcal{C}}$ satisfy $|x|,|x'| > R:=2R_{2}.$

\textbf{Step 3:} We obtain the asymptotic behaviors (\ref{asy1}) and (\ref{asy2}).

Now we consider
$$
\gamma:=\limsup _{|x| \rightarrow +\infty} \frac{u(x)}{\Gamma_{p}(x)}.
$$

We claim that  
    \begin{equation}\label{lim}
    \gamma=\lim _{R \rightarrow +\infty}\left(\sup _{\{x\in \overline{\mathcal{C}}:|x|=R \}} \frac{u(x)}{\Gamma_{p}(x)}\right). \end{equation}

    Let's prove in three cases: 
    
For the case (i) $$\gamma>\sup_{\{x\in \overline{\mathcal{C}}:|x|=R_{0} \}} \frac{u(x)}{\Gamma_{p}(x)},$$
we consider a function 
$$G(\tilde{R}_{0}):= \sup\limits_{{\{x\in \overline{\mathcal{C}}:R_{0} \leq |x|\leq \tilde{R}_{0}\}} } \frac{u(x)}{\Gamma_{p}(x)}$$ for any $\tilde{R}_{0}\geq R_{0}.$
Then $$G(\tilde{R}_{0})> \sup\limits_{\{x\in \overline{\mathcal{C}}:|x|= R_{0} \}}\frac{u(x)}{\Gamma_{p}(x)}$$
for any sufficiently large $\tilde{R}_{0}$. So by applying Corollary \ref{smp} and Remark \ref{boundeddomain}, we can get $$G(\tilde{R}_{0})=\sup\limits_{\{x\in \overline{\mathcal{C}}:|x|= \tilde{R}_{0} \}} \frac{u(x)}{\Gamma_{p}(x)}.$$ Since $G(\tilde{R}_{0})$ is an increasing function, by (\ref{3.4est1}), we can deduce (\ref{lim}).

For the case (ii) $$\gamma=\sup_{\{x\in \overline{\mathcal{C}}:|x|=R_{0} \}} \frac{u(x)}{\Gamma_{p}(x)},$$
if there exists a point $x_{0}\in \overline{\mathcal{C}}$ with $|x_{0}|>R_{0}$ such that $$\sup _{\{x\in \overline{\mathcal{C}}:|x|\geq R_{0}\}} \frac{u(x)}{\Gamma_{p}(x)}=\dfrac{u(x_{0})}{\Gamma_{p}(x_{0})},
$$then by Corollary \ref{smp} and Remark \ref{boundeddomain}, $\dfrac{u(x)}{\Gamma_{p}(x)}=$constant in $\{x\in\overline{\mathcal{C}}:|x|\geq R_{0}\}$. The claim still holds. 

Otherwise, we can find that $$\gamma=\sup\limits_{\{x\in \overline{\mathcal{C}}:|x|\geq R_{0} \}}\frac{u(x)}{\Gamma_{p}(x)}> \sup _{\{x\in \overline{\mathcal{C}}:|x|=R'_{0} \}} \frac{u(x)}{\Gamma_{p}(x)}$$ for a fixed constant $R'_{0}>R_{0}$. Then we also consider a function $$G_{1}(\tilde{R}_{0}):= \sup\limits_{{\{x\in \overline{\mathcal{C}}: R'_{0}\leq |x|\leq \tilde{R}_{0}\}} } \frac{u(x)}{\Gamma_{p}(x)}$$ for any sufficiently large $\tilde{R}_{0}\geq R'_{0}$ and the remaining proof is the same as in case (i).

For the case (iii) $$\gamma<\sup_{\{x\in \overline{\mathcal{C}}:|x|=R_{0} \}} \frac{u(x)}{\Gamma_{p}(x)},$$
we first consider a function
$$G_{2}(\tilde{R}_{0})=\sup\limits_{\{x\in \overline{\mathcal{C}}:|x|\geq \tilde{R}_{0} \}}\frac{u(x)}{\Gamma_{p}(x)}$$ for any $\tilde{R}_{0}\geq R_{0}.$

If $G_{2}(\tilde{R}_{0})>\gamma$ for any $\tilde{R}_{0}\geq R_{0},$ then since the function $G_{2}(\tilde{R}_{0})$ is nonincreasing, by using Corollary \ref{smp} and Remark \ref{boundeddomain}, we have $$G_{2}({R}_{3})=\sup _{\{x\in \overline{\mathcal{C}}:|x|= R_{3}\}} \frac{u(x)}{\Gamma_{p}(x)}\geq G_{2}({R}_{4})=\sup _{\{x\in \overline{\mathcal{C}}:|x|= R_{4}\}} \frac{u(x)}{\Gamma_{p}(x)}$$ for any $R_{0}\leq R_{3}\leq R_{4}$. 
Therefore, by (\ref{3.4est1}), we can obtain the claim above.

Otherwise, there exists a constant $\overline{R}_{0}>R_{0}$ such that $G_{2}(\overline{R}_{0})= \gamma$, then by applying the same discussion as in case (ii), we can prove the claim above.

Consequently, from (\ref{lim}), the compactness of $\{x\in \overline{\mathcal{C}}:|x|=R \}$ and the continuity of $u$, we can find points $x_{r} \in \overline{\mathcal{C}}$ whenever $r \geq  R_{0}$, such that $\left|x_{r}\right|=r$ and
$$
\lim _{r \rightarrow +\infty} \frac{u\left(x_{r}\right)}{\Gamma_{p}\left(x_{r}\right)}=\gamma.
$$

Furthermore, we consider a family of functions $\left\{u_{r}\right\}_{r \geq 2 R_{0}}$ with
$$
u_{r}(\xi):=u(r \xi) r^{\frac{n-p}{p-1}} \text { for }\xi\in \overline{\mathcal{C}}\text{ with }|\xi| > \frac{1}{2}.
$$
From the estimates (\ref{C1asy}), we can apply the Ascoli-Arzel\`{a} theorem to conclude that 
there exists a function $U=U(\xi)$ defined for $\xi\in \overline{\mathcal{C}}\text{ with }|\xi| > \frac{1}{2}$ and a sequence $r_{k} \rightarrow+\infty$ as $k \rightarrow+\infty$ such that
$$u_{r_{k}}\rightarrow U \text{ in the norm }C^{1}$$ on the compact subsets of $\{\xi\in \overline{\mathcal{C}}: |\xi| > \frac{1}{2}\}$.
And the function $U$ is $p$-harmonic on $\xi\in \overline{\mathcal{C}}\text{ with }|\xi| > \frac{1}{2}$ 
in the distribution sense. From (\ref{lim}) and
$$
\frac{u_{r}(\xi)}{\Gamma_{p}(\xi)}=\frac{u(r \xi)}{\Gamma_{p}(r \xi)},
$$
we know
$$
\frac{U(\xi)}{\Gamma_{p}(\xi)} \leq \gamma \text { for }\xi\in \overline{\mathcal{C}}\text{ with }|\xi| > \frac{1}{2}.
$$

We take $\xi_{k}=\frac{1}{r_{k}} x_{r_{k}}$ such that $\{\xi_{k}\}\subset \overline{\mathcal{C}}$. Note that $|\xi_{k}|=1$ in a compact set, and as a result, the subsequence $\left\{\xi_{k}\right\}$ converges to a point $\xi_{0}\in \overline{\mathcal{C}}$ with $\left|\xi_{0}\right|=1$ as $k \rightarrow+\infty$. Since the subsequence $u_{r_{k}}$ converges $U$ uniformly on the compact subsets, we know 
$$
\frac{U\left(\xi_{0}\right)}{\Gamma_{p}\left(\xi_{0}\right)}=\lim _{k \rightarrow+\infty} \frac{u_{r_{k}}\left(\xi_{r_{k}}\right)}{\Gamma_{p}\left(\xi_{r_{k}}\right)}=\lim_{k \rightarrow+\infty} \frac{u(x_{r_{k}})}{\Gamma_{p}\left(x_{r_{k}}\right)}=\gamma.
$$
According to Corollary \ref{smp} and Remark \ref{boundeddomain},  
$$
\frac{U(\xi)}{\Gamma_{p}(\xi)}\equiv\gamma \text { for }\xi\in \overline{\mathcal{C}}\text{ with }|\xi| > \frac{1}{2}.
$$It follows that 
$$
\lim _{r \rightarrow+\infty} u_{r}(\xi)=U(\xi)=\gamma \Gamma_{p}(\xi)$$ $\text {uniformly on the compact subsets of }\{\xi\in \overline{\mathcal{C}}: |\xi| > \frac{1}{2}\}.$

Therefore, 
$$
\lim _{r \rightarrow+\infty} \frac{u(r \xi)}{\Gamma_{p}(r \xi)}=\gamma
$$
uniformly on the $\partial B_{1}\cap \overline{\mathcal{C}}$, which implies (\ref{asy1}).

Meanwhile, since the family $u_{r}$ converges to $U$ in the norm $C^{1}$ on the compact subsets of $\{\xi\in \overline{\mathcal{C}}: |\xi| > \frac{1}{2}\}$,
we have
$$
\lim\limits_{r \rightarrow+\infty} \nabla u (r \xi) r^{\frac{n-1}{p-1}}=\lim\limits_{r \rightarrow+\infty} \nabla u_{r}(\xi)=\nabla U(\xi)=\gamma \frac{p-n}{p-1}|\xi|^{-\frac{n-1}{p-1}} \frac{\xi}{|\xi|}
$$
$\text {uniformly on the compact subsets of }\{\xi\in \overline{\mathcal{C}}: |\xi| > \frac{1}{2}\}.$
Hence,
$$
\lim _{|x| \rightarrow+\infty} (\nabla u(x)-\gamma \nabla \Gamma_{p}(x))|x|^{\frac{n-1}{p-1}}=0
,$$
which implies (\ref{asy2}).
\end{proof}

\section{The proof of Theorem \ref{thm:1.2}}\label{sec4}

To prove Theorem \ref{thm:1.2},  it is necessary to establish some integral identities.

First, we express the ${\rm Cap}_{ p}(\Omega\cap\mathcal{C};\mathcal{C})$ using the integral over $\Sigma$.

\begin{lemma}\label{lem:4.1}
Let $\Omega$ and $\overline{\Sigma}$ be as stated in Theorem \ref{thm:1.2}. Moreover, we assume that $\overline{\Sigma}$ is perpendicular to $\partial \overline{\mathcal{C}}$ along $\partial{\overline{\Sigma}}.$
Then, the solution $u$ to \eqref{eqn:1.1} satisfies
\begin{equation}\label{eqn:cap}
p{\rm Cap}_{ p}(\Omega\cap\mathcal{C};\mathcal{C})=
\int_{\Sigma}|\nabla u|^{p-1}~d \mathcal{H}^{n-1}.
\end{equation}
\end{lemma}

\begin{proof}
Let $${\rm Crit}(u):=\{x\in\overline{{\mathcal{C}}}\backslash{\Omega}:\nabla u=0\}.$$ By Theorem \ref{Asymptotic expansion} and $u\in C^{1}(\overline{\mathcal{C}}\backslash\Omega)$, we know that the set Crit$(u)$ is a compact set in $\overline{{\mathcal{C}}}\backslash{\Omega}$.
Due to the compactness of Crit$(u)$
in $\overline{{\mathcal{C}}}\backslash{\Omega}$, we have Crit$(u)\subset \{x\in \overline{{\mathcal{C}}}\backslash{\Omega}:u> t\}$ for $t > 0$ small enough.
Then we can denote by $\nu_{t}$ the exterior unit normal vector to $D_{t}=\{x\in \overline{{\mathcal{C}}}\backslash{\Omega}: u(x)>t\}$ so that $\nu_{t}=-\dfrac{\nabla u}{|\nabla u|}$ on $\partial D_{t}\cap \mathcal{C}$. 

We claim that the integral
 $$\int_{\partial D_{t}\cap \mathcal{C}} -|\nabla u|^{p-1}~ d \mathcal{H}^{n-1}$$
is independent of $0<t \leq 1$.

In fact, for sufficiently small $t>0$, due to the condition $\langle\nabla u, \nu_{\Gamma}\rangle=0$ on $\Gamma$, we know $\langle\nu_{t}, \nu_{\Gamma}\rangle=0$ on $\partial \{u=t\}$, namely $\partial D_{t}\cap\mathcal{C}$ is perpendicular to $\partial \overline{\mathcal{C}}$. So the domain $D_{t}$ is a Lipschitz domain. Then, by the equation $\Delta_{p} u=0$ in ${\Omega}^{\mathcal{C}}$ and $|\nabla u|^{p-2}\nabla u\in C^{0}(\overline{\mathcal{C}}\backslash\Omega)$, we can apply Lemma \ref{div} to obtain
$$0=
\int_{\partial D_{t}\cap \mathcal{C}} |\nabla u|^{p-2}\left\langle \nabla u , \nu_{t}\right\rangle ~d \mathcal{H}^{n-1} +\int_{\Sigma} |\nabla u|^{p-1}~ d \mathcal{H}^{n-1} $$
$$+\int_{\partial{D_{t}}\cap\Gamma }|\nabla u|^{p-1}\langle\nabla u, \nu_{\Gamma}\rangle ~d \mathcal{H}^{n-1}
$$
for any $0<t<1$. Moreover, since $\langle\nabla u, \nu_{\Gamma}\rangle=0$, we obtain
$$
\int_{\partial D_{t}\cap \mathcal{C}} -|\nabla u|^{p-1} ~d \mathcal{H}^{n-1}=\int_{\Sigma} -|\nabla u|^{p-1} ~d \mathcal{H}^{n-1}
$$
for any $0<t<1$.

Then, by using the identity (\ref{eq:cap}) and the coarea formula we have
\begin{equation*}
\begin{aligned}
p \operatorname{{\rm Cap}}_{p}(\Omega\cap\mathcal{C};\mathcal{C})&=\int_{{\Omega}^{\mathcal{C}}} |\nabla u|^{p}~ d x=\int_{0}^{1} \int_{\partial D_{t}\cap \mathcal{C}}  |\nabla u|^{p-1} ~d \mathcal{H}^{n-1} d t \\ &=\int_{\Sigma}|\nabla u|^{p-1} ~d \mathcal{H}^{n-1}.
\end{aligned}
\end{equation*}
This completes the proof of Lemma \ref{lem:4.1}.
\end{proof}

\begin{remark}
    In contrast to the paper \cite{BC}, we need to consider the critical points. In fact, Appendix A in the paper \cite{BC} discusses the case of a convex domain in $\mathbb{R}^{n}$  for the corresponding result above, and it is proven in Appendix B that there are no critical points in that case.
\end{remark}

Next, we calculate the value of $\gamma$. This will help us compare the value of the $P$ function at infinity with the value at the boundary $\Sigma$.
\begin{corollary}\label{gamma}
    The constant $\gamma$ in Theorem \ref{Asymptotic expansion} is given by 
    $$(\frac{p}{n\omega^{\mathcal{C}}_{n}})^{\frac{1}{p-1}}\frac{p-1}{n-p}{\rm Cap}_{p}(\Omega\cap\mathcal{C};\mathcal{C})^{\frac{1}{p-1}},$$
    where $\omega^{\mathcal{C}}_{n}:=\mathcal{H}^{n}(\mathcal{C}\cap B_{1})$ is the volume of intersection between the unit ball and the convex cone $\mathcal{C}$.
\end{corollary}
 
\begin{proof}
    
By $\langle\nabla u,\nu_{\Gamma}\rangle=0$ on $\Gamma$ and applying the divergence theorem \ref{div} to the equation $\Delta_{p}u=0$ on the bounded open set $B_{R}\backslash \Omega^{\mathcal{C}}$ for large $R$, we have $$0=\int_{\Sigma} |\nabla u|^{p-2}\langle\nabla u, -\nu_{\Sigma}\rangle ~d \mathcal{H}^{n-1}+\int_{\Gamma_{R}\cup
(B_{R}\cap\Gamma)}|\nabla u|^{p-2}\langle\nabla u, \nu\rangle ~d \mathcal{H}^{n-1}$$
$$=\int_{\Sigma} |\nabla u|^{p-2}\langle\nabla u, -\nu_{\Sigma}\rangle ~d \mathcal{H}^{n-1}+\int_{\Gamma_{R}}|\nabla u|^{p-2}\langle\nabla u, \nu\rangle ~d \mathcal{H}^{n-1},$$
where the unit normal vector $\nu_{\Sigma}=-\dfrac{\nabla u}{|\nabla u|}$ points towards ${\Omega}^{\mathcal{C}}$. 

Then, from Lemma $\ref{lem:4.1}$ and the asymptotic behavior (\ref{asy2}) of $u$, we can obtain 
$$p{\rm Cap}_{p}(\Omega\cap\mathcal{C};\mathcal{C})=n\omega^{\mathcal{C}}_{n}\gamma^{p-1}(\frac{n-p}{p-1})^{p-1}.$$
\end{proof}

Then, we prove the following Rellich-Poho\v{z}aev-type identity.
\begin{lemma}\label{lem:4.2}
Let $\Omega$ and $\overline{\Sigma}$ be as stated in Theorem \ref{thm:1.2}. Moreover, we assume that $\overline{\Sigma}$ is perpendicular to $\partial \overline{\mathcal{C}}$ along $\partial{\overline{\Sigma}}.$
Then, the solution $u$ to \eqref{eqn:1.1} satisfies
\begin{equation}\label{eqn:rp}
(p-1)\int_{\Sigma} |\nabla u|^{p}\langle x , \nu_{\Sigma}\rangle~ d \mathcal{H}^{n-1}=(n-p)p {\rm Cap}_{p}(\Omega\cap\mathcal{C};\mathcal{C}),
\end{equation}
where the unit normal vector $\nu_{\Sigma}=-\dfrac{\nabla u}{|\nabla u|}$ points towards ${\Omega}^{\mathcal{C}}$.
\end{lemma}

\begin{proof}
From the Poho\v{z}aev type identity (Theorem 4.1 in \cite{CL}) on the bounded open set $B_{R}\cap{\Omega}^{\mathcal{C}}$,
by $\langle\nabla u,\nu_{\Gamma}\rangle=0$ on $\Gamma$,  $\langle x, \nu_{\Gamma}\rangle=0$ on $\Gamma$ and direct computations, we get, for $R$ sufficiently large,
\begin{equation*}
\begin{aligned}
    &\dfrac{p-n}{p}\int_{B_{R}\cap{\Omega}^{\mathcal{C}}}|\nabla u|^{p} ~dx\\
=&\int_{(B_{R}\cap\Gamma)\cup\Sigma\cup(\partial B_{R}\cap\mathcal{C})}(|\nabla u|^{p-2}\langle\nabla u,x\rangle\langle\nabla u,\nu\rangle-\frac{1}{p}|\nabla u|^{p}\langle x,\nu\rangle ) ~d \mathcal{H}^{n-1}\\
=&\frac{p-1}{p}\int_{\Sigma}|\nabla u|^{p}\langle x,-\nu_{\Sigma}\rangle ~d \mathcal{H}^{n-1}\\
&+\int_{\partial B_{R}\cap\mathcal{C}} (|\nabla u|^{p-2}\langle\nabla u,x\rangle\langle\nabla u,\nu\rangle-\frac{1}{p}|\nabla u|^{p}\langle x,\nu\rangle)~ d \mathcal{H}^{n-1}.
\end{aligned}
\end{equation*}
Then, by taking the limit for $R\rightarrow+\infty$ and noting that the integrals on $\partial B_R\cap\mathcal{C}$ converge to zero due to the asymptotic behavior of $u$ at infinity given by Theorem \ref{Asymptotic expansion}. Thus, we obtain
the assertion.
\end{proof}

Now, using the identities above, we can calculate the value $C$ of $|\nabla u|$ on $\Sigma$ with the overdetermined condition \eqref{Overdeterminedcondition}. At the same time, an expression for ${\rm Cap}_{p}(\Omega\cap\mathcal{C};\mathcal{C})$ that is only related to the geometric quantity of the domain $\Omega\cap\mathcal{C}$ and independent of the solution $u$ can also be given.
\begin{proposition}\label{4.4}
Let $\Omega$ and $\overline{\Sigma}$ be as stated in Theorem \ref{thm:1.2}. Moreover, we assume that $\overline{\Sigma}$ is perpendicular to $\partial \overline{\mathcal{C}}$ along $\partial{\overline{\Sigma}}.$
 Let $u$ be a weak solution to \eqref{eqn:1.1} and \eqref{Overdeterminedcondition}. The constant $C$ appearing in \eqref{Overdeterminedcondition} equals

\begin{equation}\label{eqn:VC}
C=\dfrac{n-p}{n(p-1)}\dfrac{P(\Omega; \mathcal{C})}{\mathcal{H}^{n}(\Omega\cap{\mathcal{C}})}.
\end{equation}

Moreover, the following explicit expression of the  $p$-capacity of $\Omega\cap\mathcal{C}$
holds:
\begin{equation}\label{eqn:24}
{\rm Cap}_{p}(\Omega\cap\mathcal{C};\mathcal{C})=\frac1p\Big(\frac{n-p}{p-1}\Big)^{p-1}\frac{P(\Omega; \mathcal{C})^{p}}{(n\mathcal{H}^{n}(\Omega\cap{\mathcal{C}}))^{p-1}}.
\end{equation}
\end{proposition}
\begin{proof}
By using the identities \eqref{eqn:cap} and \eqref{eqn:rp}, we can separately obtain the following equations
\begin{equation*}
{\rm Cap}_{p}(\Omega\cap\mathcal{C};\mathcal{C})=\frac1pC^{p-1}P(\Omega; \mathcal{C}) \ \ {\rm and}\ \ {\rm Cap}_{p}(\Omega\cap\mathcal{C};\mathcal{C})=\frac{n(p-1)}{p(n-p)}C^p\mathcal{H}^{n}(\Omega\cap{\mathcal{C}}).
\end{equation*}
Combining these equations, we can obtain equations \eqref{eqn:VC} and \eqref{eqn:24}.
\end{proof}

Next, we apply the strong maximum principle to the $P$-function.

{\bf The $P$-function.} We introduce the $P$-function defined by
\begin{equation}\label{eqn:4.1}
  P={u^{-\frac{p(n-1)}{n-p}}}{|\nabla u|^{p}}.
\end{equation}
From the $C^{1}(\overline{\mathcal{C}}\backslash{\Omega})$ regularity of the solution $u$ in Theorem \ref{regularity} and $u>0$, we know that the $P$-function is $C^{0}(\overline{\mathcal{C}}\backslash{\Omega}).$

\begin{proposition}\label{Pfunction}
Let $\Omega$ and $\overline{\Sigma}$ be as stated in Theorem \ref{thm:1.2}. Moreover, we assume that $\overline{\Sigma}$ is perpendicular to $\partial \overline{\mathcal{C}}$ along $\partial{\overline{\Sigma}}.$
Let $u$ be a weak solution to the problem (\ref{eqn:1.1}). Then, at $\{ \nabla u\not=0\}$,
$$a_{ij,p}(\nabla u)P_{ij}+L_iP_i\geq0,$$ where $L_{i}P_{i}$ is the lower order term of $P_{i}$. 
Moreover, the function $P$ can not attain a maximum at any point of $\Gamma$ and any interior point of\ ${\Omega}^{\mathcal{C}}$, unless
$P$ is a constant.
\end{proposition}
\begin{proof}
Let $\mathring{\rm Crit}(u):=\{x\in{\Omega}^{\mathcal{C}}:\nabla u=0\}$. The calculations are all taken in
${\Omega}^{\mathcal{C}}\backslash{{\mathring{\rm Crit}(u)}}$. Since the solution $u$ is a smooth function in ${\Omega}^{\mathcal{C}}\backslash{{\rm Crit}}(u)$, we can calculate the derivative in the classic sense. The detailed calculations to get the equation $a_{ij,p}P_{ij}+L_iP_i\geq0$ can be referred to \cite[Proposition 3.8]{XY}.

Suppose that there exists one point $x_{0}\in\Gamma$ such that $P(x_{0})=\sup\limits_{{\Omega}^{\mathcal{C}}} P>0.$ So there is no critical point in the neighborhood of $x_{0}$. Then we can use the Hopf lemma to obtain $P_{\nu_{\Gamma}}(x_{0})>0.$

On the other hand, by direct computations,
\begin{equation*}
\begin{aligned}
 \langle\nabla P,\nu_{\Gamma}\rangle(x_{0})&=
-p|\nabla u|^{p-1}u^{\frac{p(1-n)}{n-p}}\bigcdot\\
[-\nabla^{2}u(\frac{\nabla u}{|\nabla u|}&,\nu_{\Gamma})+\frac{n-1}{n-p}\dfrac{|\nabla u|}{u}\langle\nabla u,\nu_{\Gamma}\rangle].
\end{aligned}
\end{equation*}
 From $\langle\nabla u,\nu_{\Gamma}\rangle=0$ on $\Gamma$, if we denote $\nabla_{\Gamma}u$ as the tangential component of $\nabla u$, then we deduce that at $x_{0}$,
    $\nabla u=\nabla_{\Gamma} u$ and
   $$0=\nabla_{\Gamma}u(\langle\nabla u,\nu_{\Gamma}\rangle)=\nabla^{2}u(\nabla u,\nu_{\Gamma})+h(\nabla_{\Gamma} u,\nabla_{\Gamma} u).$$ 
By the convexity of cone $\Gamma$, we obtain $\nabla^{2} u(\nabla u,\nu_{\Gamma})\leq 0.$ Thus, we know $\langle\nabla P,\nu_{\Gamma}\rangle(x_{0})\leq 0.$
So the maximum value can not be achieved at the boundary $\Gamma$.

Suppose that there exists one interior point $x_{1}\in{\Omega}^{\mathcal{C}}$ such that $P(x_{1})=\sup\limits_{{\Omega}^{\mathcal{C}}} P>0.$ This means that $|\nabla u(x_{1})|\neq 0.$ So the equation is non-degenerate near the point $x_{1}.$ Then by applying the maximum principle, we know that $P$ is constant near the point $x_{1}$. Thus $P$ is constant on ${\Omega}^{\mathcal{C}}$. 
\end{proof}

\begin{proof}[\textbf{Proof of Theorem \ref{thm:1.2}}]
By applying Theorem \ref{Asymptotic expansion} and Corollary \ref{gamma}, we can  obtain $$\lim\limits_{|x|\rightarrow +\infty} P=(\frac{n\omega^{\mathcal{C}}_{n}}{p})^{\frac{p}{n-p}}(\frac{n-p}{p-1})^{\frac{p(n-1)}{n-p}}({\rm Cap}_{p}(\Omega\cap\mathcal{C};\mathcal{C})^{\frac{-p}{n-p}}.$$
Moreover, by Proposition \ref{4.4}, we have
$$\lim\limits_{|x|\rightarrow +\infty} P=(n\omega^{\mathcal{C}}_{n})^{\frac{p}{n-p}}(\frac{n-p}{p-1})^{p}(\frac{(n\mathcal{H}^{n}(\Omega\cap{\mathcal{C}}))^{p-1}}{P(\Omega; \mathcal{C})^{p}})^{\frac{p}{n-p}}$$
and
 $$P|_{\Sigma}=(\dfrac{n-p}{n(p-1)})^{p}(\dfrac{P(\Omega; \mathcal{C})}{\mathcal{H}^{n}(\Omega\cap{\mathcal{C}})})^{p}.$$ 
According to the isoperimetric inequality in the convex cone (Theorem \ref{isoperimetric}), we have \begin{equation}\label{compinfinity}
\lim\limits_{|x|\rightarrow +\infty} P\leq P|_{\overline{\Sigma}}.
\end{equation}
Here, if the equal sign holds, from the rigidity result in Theorem \ref{isoperimetric}, we have completed the proof.

From Theorem \ref{regularity} and the fact $u>0$, the $P$-function is continuous in $\overline{\mathcal{C}}\backslash{\Omega},$
 then we can apply the $u=1$ on $\Sigma$ and $|\nabla u|=C$ on $\Sigma$ to get $P|_{\overline{\Sigma}}=P|_{\Sigma}=constant$.

By the Proposition \ref{Pfunction}, we know that either $P$ is a constant in $\overline{\mathcal{C}}\backslash\Omega$ or $P$ attains its maximum on $\overline{\Sigma}.$ Since $P=constant$ on $\overline{\Sigma}$ in both cases, we have
  $$\langle\nabla P,\nu_{\Sigma}\rangle\leq 0$$
on $\Sigma$. Note that the unit normal vector $\nu_{\Sigma}=-\dfrac{\nabla u}{|\nabla u|}$ points towards ${\Omega}^{\mathcal{C}}$. 

By direct computations and the equality (\ref{mean curvature}), we have
\begin{equation*}
\begin{aligned}
0&\geq \langle\nabla P,\nu_{\Sigma}\rangle\\
&=-p|\nabla u|^{p-1}u^{\frac{p(1-n)}{n-p}}[\nabla^{2}u(\frac{\nabla u}{|\nabla u|},\frac{\nabla u}{|\nabla u|})-\frac{n-1}{n-p}\dfrac{|\nabla u|^{2}}{u}]\\
&=-\frac{p}{p-1}|\nabla u|^{p}u^{\frac{p(1-n)}{n-p}}[H_{\Sigma}-\frac{(n-1)(p-1)}{n-p}\dfrac{|\nabla u|}{u}].
\end{aligned}
\end{equation*}

Hence, since the boundary $\partial \Omega$ is  $ C^{2,\alpha}$, we have \begin{equation}\label{Mineq}
    H_{\overline{\Sigma}}\geq \dfrac{(n-1)(p-1)}{n-p}C=(n-1)\dfrac{P(\Omega; \mathcal{C})}{n\mathcal{H}^{n}(\Omega\cap{\mathcal{C}})}>0.
\end{equation}
By integrating the above inequality, we obtain $$\int_{\overline{\Sigma}}\dfrac{1}{H_{\overline{\Sigma}}}~ d \mathcal{H}^{n-1} \leq \dfrac{n}{n-1}\mathcal{H}^{n}(\Omega\cap{\mathcal{C}}).$$ Therefore, by applying the Heintze-Karcher type inequality in Lemma \ref{HK}, we conclude that $\overline{\Sigma}$ is a spherical cap and the radius $R=\dfrac{n-p}{p-1}C^{-1}$. Simultaneously, by combining the uniqueness of the solution, we can complete the proof of Theorem \ref{thm:1.2}.
\end{proof}

\quad

\begin{acknow}
The authors would like to thank Chao Qian, Chao Xia, Jinyu Guo and M. Fogagnolo for their helpful conversations on this work. In particular, special thanks are extended to Xiaoliang Li for his valuable suggestions.

 This work is partially supported by the National Natural Science Foundation of China (Grants Nos.~11831005, 12061131014 and 12201138) and Mathematics Tianyuan fund project (Grant No. 12226350). 
\end{acknow}

\end{document}